\documentclass[12pt]{article}
\usepackage{amsmath}
\usepackage{amssymb}
\usepackage{amsthm}
\usepackage{extarrows}

\newtheorem{theorem}{Theorem}
\newtheorem{corollary}{Corollary}

\newtheorem{definition}{Definition}

\newtheorem{proposition}{Proposition}

\newcommand{\Mk}{\mathcal{M}_k}
\newcommand{\Mkk}{\mathcal{M}_{k-1}}
\newcommand{\Hk}{\mathcal{H}_k}
\newcommand{\Rm}{\mathbb{R}^m}
\newcommand{\Clm}{\mathcal{C}l_m}
\newcommand{\C}{\mathbb{C}}
\newcommand{\Sm}{\mathbb{S}^{m-1}}
\newcommand{\scs}{\mathcal{S}}

\begin{document}

\title{On Some Conformally Invariant Operators in Euclidean Space}

\author{Chao Ding and John Ryan\\
\emph{\small Department of Mathematics, University of Arkansas, Fayetteville, AR 72701, USA}}

\date{}

\maketitle

%%%%%%%%%%%%%%%       abstract           %%%%%%%%%%%%
\begin{abstract}
The aim of this paper is to correct a mistake in earlier work on the conformal invariance of Rarita-Schwinger operators and use the method of correction to develop properties of some conformally invariant operators in the Rarita-Schwinger setting. We also study properties of some other Rarita-Schwinger type operators, for instance, twistor operators and dual twistor operators.  This work is also intended as an attempt to motivate the study of Rarita-Schwinger operators via some representation theory. This calls for a review of earlier work by Stein and Weiss.
\end{abstract}
{\bf Keywords:}\quad Irreducible representations, Stein-Weiss type operators, Rarita-Schwinger type operators, Almansi-Fischer decomposition, Iwasawa decomposition, Conformal invariance, Integral formulas.

%%%%%%%%%              Introduction            %%%%%%%%%%%%%
\section{Introduction}\hspace*{\fill} \\
\par
In representation theory for Lie groups one is interested in irreducible representation spaces. In particular, for the group $SO(m)$ one might consider the representation space of all harmonic functions on $\Rm$. This space is invariant under the action of $O(m)$, but this space is not irreducible. It decomposes into the infinite sum of harmonic polynomials each homogeneous of degree $k$, $1<k<\infty$. Each of these spaces is irreducible for $SO(m)$. See for instance \cite{G}. Hence, one may consider functions $f:U\longrightarrow \Hk$ where $U$ is a domain in $\Rm$ and $\Hk$ is the space of real valued harmonic polynomials homogeneous of degree $k$. If $\Hk$ is the space of Clifford algebra valued harmonic polynomials homogeneous of degree $k$, then an Almansi-Fischer decomposition result tells us that 
$$\Hk=\Mk\oplus u\Mkk.$$
Here $\Mk$ and $\Mkk$ are spaces of Clifford algebra valued polynomials homogeneous of degree $k$ and $k-1$ in the variable $u$, respectively and are solutions to the Dirac equation $D_uf(u)=0,$ where $D_u$ is the Euclidean Dirac operator. The elements of these spaces are known as homogeneous \emph{monogenic} polynomials. In this case the underlying group $SO(m)$ is replaced by its double cover $Spin(m)$. See \cite{Br}.
\\
\par
Classical Clifford analysis is the study of and applications of Dirac type operators. In this case, the functions considered take values in the spinor space, which is an irreducible representation of $Spin(m)$. If we replace the spinor space with some other irreducible representations, for instance, $\Mk$, we will get the Rarita-Schwinger operator as the first generalization of the Dirac operator in higher spin theory. See, for instance \cite{B}. The conformal invariance of this operator, its fundamental solutions and some associated integral formulas were first provided in \cite{B},  and then \cite{D}. However, some proofs in \cite{D} rely on the mistake that the Dirac operator in the Rarita-Schwinger setting is also conformally invariant. This will be explained and corrected in Section 3.\\
\par
From the construction of the Rarita-Schwinger operators, we notice that some other Rarita-Schwinger type operators can be constructed similarly, for instance, twistor operators, dual twistor operators and the remaining operators, see \cite{B,D,Li} . It is worth pointing out that we need to be careful for the reasons we mentioned above when we establish properties for Rarita-Schwinger type operators. Hence, we give the details of proofs of some properties and integral operators for Rarita-Schwinger type operators.
\\
\par
This paper is organized as follows: after a brief introduction to Clifford algebras and Clifford analysis in Section 2, representation theory of the Spin group and Stein-Weiss operators are used to motivate Dirac operators and Rarita-Schwinger operators. On the one hand the Dirac operator can be introduced and motivated by an adapted version of Stokes' Theorem. See \cite{E}. Motivation for Rarita-Schwinger operators seem better suited via representation theory, particularly for spin and special orthogonal groups. In Section 3, we will use a counter-example to show that the Dirac operator is not conformally invariant in the Rarita-Schwinger setting. Then we give a proof of conformal invariance of the Rarita-Schwinger operators and we provide the intertwining operators for the Rarita-Schwinger operators.
Motivated by the Almansi-Fischer decomposition mentioned above, using similar construction with the Rarita-Schwinger operator, we can consider conformally invariant operators between $\Mk$-valued functions and $u\Mkk$-valued functions. This idea brings us other Rarita-Schwinger type operators, for instance, twistor and dual twistor operators. More details of the construction and properties of these operators can be found in Section 4.
\section*{Acknowledgement}
The authors wish to thank the referee for helpful suggestions that improved the manuscript. The authors are also grateful to Bent \O rsted for communications pointing out that the intertwining operators for the Rarita-Schwinger operators are special cases of Knapp-Stein intertwining operators in higher spin theory (\cite{CO,KS}).
%%%%%%%%%%%         Preliminaries       %%%%%%%%%%%%%%%%%%%
\section{Preliminaries}
\subsection{Clifford algebra}\hspace*{\fill} \\
\par
A real Clifford algebra, $\mathcal{C}l_{m},$ can be generated from $\mathbb{R}^m$ by considering the
relationship $$\underline{x}^{2}=-\|\underline{x}\|^{2}$$ for each
$\underline{x}\in \mathbb{R}^m$.  We have $\mathbb{R}^m\subseteq Cl_{m}$. If $\{e_1,\ldots, e_m\}$ is an orthonormal basis for $\mathbb{R}^m$, then $\underline{x}^{2}=-\|\underline{x}\|^{2}$ tells us that $$e_i e_j + e_j e_i= -2\delta_{ij},$$ where $\delta_{ij}$ is the Kronecker delta function. Similarly, if we replace $\Rm$ with $\C^m$ in the previous definition and consider the relationship
$$z^2=-||z||^2=-z_1^2-z_2^2-\cdots-z_m^2,\ where\ z=(z_1,z_2,\cdots,z_m)\in\C ^m,$$
 we get complex Clifford algebra $\Clm (\C)$, which can also be defined as the complexification of the real Clifford algebra
$$\Clm (\C)=\Clm\otimes\C.$$
\par
In this paper, we deal with the real Clifford algebra $\Clm$ unless otherwise specified. An arbitrary element of the basis of the Clifford algebra can be written as $e_A=e_{j_1}\cdots e_{j_r},$ where $A=\{j_1, \cdots, j_r\}\subset \{1, 2, \cdots, m\}$ and $1\leq j_1< j_2 < \cdots < j_r \leq m.$
Hence for any element $a\in \Clm$, we have $a=\sum_Aa_Ae_A,$ where $a_A\in \mathbb{R}$. We will need the following anti-involutions:
\begin{itemize}
\item Reversion:\\
\begin{eqnarray*}
\tilde{a}=\sum_{A} (-1)^{|A|(|A|-1)/2}a_Ae_A,
\end{eqnarray*}
where $|A|$ is the cardinality of $A$. In particular, $\widetilde{e_{j_1}\cdots e_{j_r}}=e_{j_r}\cdots e_{j_1}$. Also $\widetilde{ab}=\tilde{b}\tilde{a}$ for $a,\ b\in\mathcal{C}l_m$.
\item Clifford conjugation:\\
\begin{eqnarray*}
\bar{a}=\sum_{A} (-1)^{|A|(|A|+1)/2}a_Ae_A,
\end{eqnarray*}
satisfying $\overline{e_{j_1}\cdots e_{j_r}}=(-1)^re_{j_r}\cdots e_{j_1}$ and $\overline{ab}=\bar{b}\bar{a}$ for $a,\ b\in\mathcal{C}l_m$.
\end{itemize}

The Pin and Spin groups play an important role in Clifford analysis. The Pin group can be defined as $$Pin(m)=\{a\in \mathcal{C}l_m: a=y_1y_2\dots y_{p},\ where\ y_1,\dots,y_{p}\in\mathbb{S}^{m-1},\ p\in\mathbb{N}\},$$ 
where $\mathbb{S} ^{m-1}$ is the unit sphere in $\Rm$. $Pin(m)$ is clearly a group under multiplication in $\mathcal{C}l_m$. \\
\par
Now suppose that $a\in \mathbb{S}^{m-1}\subseteq \mathbb{R}^m$, if we consider $axa$, we may decompose
$$x=x_{a\parallel}+x_{a\perp},$$
where $x_{a\parallel}$ is the projection of $x$ onto $a$ and $x_{a\perp}$ is the rest, perpendicular to $a$. Hence $x_{a\parallel}$ is a scalar multiple of $a$ and we have
$$axa=ax_{a\parallel}a+ax_{a\perp}a=-x_{a\parallel}+x_{a\perp}.$$
So the action $axa$ describes a reflection of $x$ across the hyperplane perpendicular to $a$. By the Cartan-Dieudonn$\acute{e}$ Theorem each $O\in O(m)$ is the composition of a finite number of reflections. If $a=y_1\cdots y_p\in Pin(m),$ we have $\tilde{a}=y_p\cdots y_1$ and observe that $ax\tilde{a}=O_a(x)$ for some $O_a\in O(m)$. Choosing $y_1,\ \dots,\ y_p$ arbitrarily in $\mathbb{S}^{m-1}$, we see that the group homomorphism
\begin{eqnarray}
\theta:\ Pin(m)\longrightarrow O(m)\ :\ a\mapsto O_a,
\end{eqnarray}
with $a=y_1\cdots y_p$ and $O_ax=ax\tilde{a}$ is surjective. Further $-ax(-\tilde{a})=ax\tilde{a}$, so $1,\ -1\in Ker(\theta)$. In fact $Ker(\theta)=\{1,\ -1\}$. See \cite{P1}. The Spin group is defined as
$$Spin(m)=\{a\in \mathcal{C}l_m: a=y_1y_2\dots y_{2p},y_1,\dots,y_{2p}\in\mathbb{S}^{m-1},\ p\in\mathbb{N}\}$$
 and it is a subgroup of $Pin(m)$. There is a group homomorphism
\begin{eqnarray*}
\theta:\ Spin(m)\longrightarrow SO(m)\ ,
\end{eqnarray*}
which is surjective with kernel $\{1,\ -1\}$. It is defined by $(1)$. Thus $Spin(m)$ is the double cover of $SO(m)$. See \cite{P1} for more details.\\
\par
For a domain $U$ in $\Rm$, a diffeomorphism $\phi: U\longrightarrow \mathbb{R}^m$ is said to be conformal if, for each $x\in U$ and each $\bold{u,v}\in TU_x$, the angle between $\bold{u}$ and $\bold{v}$ is preserved under the corresponding differential at x, $d\phi_x$.
For $m\geq 3$, a theorem of Liouville tells us the only conformal transformations are M\"obius transformations. Ahlfors and Vahlen show that given a M\"{o}bius transformation on $\mathbb{R}^m \cup \{\infty\}$ it can be expressed as $y=(ax+b)(cx+d)^{-1}$ where $a,\ b,\ c,\ d\in \Clm$ and satisfy the following conditions \cite{L1}:
\begin{eqnarray*}
&&1.\ a,\ b,\ c,\ d\ are\ all\ products\ of\ vectors\ in\ \mathbb{R}^m;\\
&&2.\ a\tilde{b},\ c\tilde{d},\ \tilde{b}c,\ \tilde{d}a\in\mathbb{R}^m;\\
&&3.\ a\tilde{d}-b\tilde{c}=\pm 1.
\end{eqnarray*}
Since $y=(ax+b)(cx+d)^{-1}=ac^{-1}+(b-ac^{-1}d)(cx+d)^{-1}$, a conformal transformation can be decomposed as compositions of translation, dilation, reflection and inversion. This gives an \emph{Iwasawa decomposition} for M\"obius transformations. See \cite{Li} for more details. In Section 3, we will show that the Rarita-Schwinger operator is conformally invariant.\\
\par
The Dirac operator in $\mathbb{R}^m$ is defined to be $$D_x:=\displaystyle\sum_{i=1}^{m}e_i\partial_{x_i}.$$ 
 We also let $D$ denote the Dirac operator if there is no confusion in which variable it is with respect to. Note $D_x^2=-\Delta_x$, where $\Delta_x$ is the Laplacian in $\mathbb{R}^m$.  A $\Clm$-valued function $f(x)$ defined on a domain $U$ in $\Rm$ is called left monogenic if $D_xf(x)=0.$ Since multiplication of Clifford numbers is not commutative, there is a similar definition for right monogenic functions.\\
\par
Let $\mathcal{M}_k$ denote the space of $\mathcal{C}l_m$-valued monogenic polynomials, homogeneous of degree $k$. Note that if $h_k\in\Hk$, the space of $\mathcal{C}l_m$-valued harmonic polynomials homogeneous of degree $k$, then $Dh_k\in\mathcal{M}_{k-1}$, but $Dup_{k-1}(u)=(-m-2k+2)p_{k-1}u,$ so
$$\mathcal{H}_k=\mathcal{M}_k\oplus u\mathcal{M}_{k-1},\ h_k=p_k+up_{k-1}.$$
This is an \emph{Almansi-Fischer decomposition} of $\Hk$. See \cite{D} for more details. Similarly, we can obtain by conjugation a right Almansi-Fischer decomposition,
$$\mathcal{H}_k=\overline{\mathcal{M}}_k\oplus \overline{\mathcal{M}}_{k-1}u,$$
where $\overline{\mathcal{M}}_k$ stands for the space of right monogenic polynomials homogeneous of degree $k$.\\
\par
In this Almansi-Fischer decomposition, we define $P_k$ as the projection map 
\begin{eqnarray*}
P_k: \mathcal{H}_k\longrightarrow \mathcal{M}_k.
\end{eqnarray*}
Suppose $U$ is a domain in $\mathbb{R}^m$. Consider $f: U\times \mathbb{R}^m\longrightarrow \mathcal{C}l_m,$
such that for each $x\in U$, $f(x,u)$ is a left monogenic polynomial homogeneous of degree $k$ in $u$, then the Rarita-Schwinger operator is defined as follows
 $$R_k:=P_kD_xf(x,u)=(\frac{uD_u}{m+2k-2}+1)D_xf(x,u).$$
We also have a right projection $P_{k,r}:\ \Hk\longrightarrow\overline{\mathcal{M}}_k$, and a right Rarita-Schwinger operator $R_{k,r}=D_xP_{k,r}.$ See \cite{B,D}.

%%%%%%%%%%%%%%%%%%%%        Irreducible representations of Spin group            %%%%%%%%%%%%%%%%%%%%
\subsection{Irreducible representations of the Spin group}\hspace*{\fill} \\
\par
To motivate the Rarita-Schwinger operators and to be  relatively self-contained we cover in the rest of Section 2 some basics on representation theory.
\begin{definition}
A Lie group is a smooth manifold $G$ which is also a group such that multiplication $(g,h)\mapsto gh\ :\ G\times G\longrightarrow G$ and inversion $g\mapsto g^{-1}\ :\ G\longrightarrow G$ are both smooth.
\end{definition}
Let $G$ be a Lie group and $V$ a vector space over $\mathbb{F}$, where $\mathbb{F}=\mathbb{R}$ or $\mathbb{C}$. A $\emph{representation}$ of $G$ is a pair $ (V,\tau)$ in which $\tau$ is a homomorphism from $G$ into the group $Aut(V)$ of invertible $\mathbb{F}$-linear transformations on $V$. Thus $\tau(g)$ and its inverse $\tau(g)^{-1}$ are both $\mathbb{F}$-linear operators on $V$ such that
$$\tau(g_1g_2)=\tau(g_1)\tau(g_2),\ \ \ \ \tau(g^{-1})=\tau(g)^{-1}$$
for all $g_1,\ g_2$ and $g$ in $G$. In practice, it will often be convenient to think and speak of $V$ as simply a $\emph{G-module}$. A subspace $U$ in $V$ which is \emph{G-invariant} in the sense that $gu\in U$ for all $g\in G$ and $u\in U$, is called a \emph{submodule} of $V$ or a \emph{subrepresentation}. The dimension of $V$ is called the dimension of the representation. If $V$ is finite-dimensional it is said to be \emph{irreducible} when it contains no submodules other than $0$ and itself; otherwise, it is said to be \emph{reducible}. The following three representation spaces of the Spin group are frequently used in Clifford analysis.

%%%%%%%%          spinor representation         %%%%%%%%%%%%%
\subsubsection{Spinor representation space $\mathcal{S}$}
The most commonly used representation of the Spin group in $\mathcal{C}l_m(\C)$ valued function theory is the spinor space. The construction is as follows:\\
Let us consider complex Clifford algebra $\mathcal{C}l_m(\mathbb{C})$ with even dimension $m=2n$. $\C^m$ or the space of vectors is embedded in $\Clm(\C)$ as
\begin{eqnarray*}
(x_1,x_2,\cdots,x_m)\mapsto \sum^{m}_{j=1}x_je_j:\ \C ^m\hookrightarrow \mathcal{C}l_m(\C).
\end{eqnarray*}
Define the \emph{Witt basis} elements of $\C^{2n}$ as 
$$f_j:=\displaystyle\frac{e_j-ie_{j+n}}{2},\ \ f_j^{\dagger}:=-\displaystyle\frac{e_j+ie_{j+n}}{2}.$$
Let $I:=f_1f_1^{\dagger}\dots f_nf_n^{\dagger}$. The space of \emph{Dirac spinors} is defined as
$$\mathcal{S}:=\Clm(\C)I.$$ This is a representation of $Spin(m)$ under the following action
$$\rho(s)I:=sI,\ for\ s\in Spin(m).$$
Note that $\scs$ is a left ideal of $\Clm (\C)$. For more details, we refer the reader to \cite{De}. An alternative construction of spinor spaces is given in the classical paper of Atiyah, Bott and Shapiro \cite{At}.

%%%%%%%%       Harmonic polynomials      %%%%%%%%%%%%%
\subsubsection{Homogeneous harmonic polynomials on $\mathcal{H}_k(\Rm,\mathbb{C})$}\hspace*{\fill} \\
\par

It is a well-known fact that the space of harmonic polynomials is invariant under the action of $Spin(m)$, since the Laplacian $\Delta_m$ is an $SO(m)$ invariant operator. But it is not irreducible for $Spin(m)$. It can be decomposed into the infinite sum of $k$-homogeneous harmonic polynomials, $1<k<\infty$. Each of these spaces is irreducible for $Spin(m)$. This brings us the most familiar representations of $Spin(m)$: spaces of $k$-homogeneous harmonic polynomials on $\mathbb{R}^m$. The following action has been shown to be an irreducible representation of $Spin(m)$ (see \cite{L}): 
\begin{eqnarray*}
\rho\ :\ Spin(m)\longrightarrow Aut(\Hk),\ s\longmapsto \big(f(x)\mapsto \tilde{s}f(sx\tilde{s})s\big).
\end{eqnarray*}
This can also be realized as follows
\begin{eqnarray*}
Spin(m)\xlongrightarrow{\theta}SO(m)\xlongrightarrow{\rho} Aut(\Hk);\\
a\longmapsto O_a\longmapsto \big(f(x)\mapsto f(O_ax)\big),
\end{eqnarray*}
where $\theta$ is the double covering map and $\rho$ is the standard action of $SO(m)$ on a function $f(x)\in\Hk$ with $x\in\mathbb{R}^m$.

%%%%%%%%%%           Monogenic polynomials        %%%%%%%%%%%%%%
\subsubsection{Homogeneous monogenic polynomials on $\mathcal{C}l_m$}\hspace*{\fill} \\
\par
In $\mathcal{C}l_m$-valued function theory, the previously mentioned Almansi-Fischer decomposition shows us we can also decompose the space of $k$-homogeneous harmonic polynomials as follows
$$\Hk=\Mk\oplus u\Mkk.$$
If we restrict $\Mk$ to the spinor valued subspace, we have another important representation of $Spin(m)$: the space of $k$-homogeneous spinor-valued monogenic polynomials on $\Rm$, henceforth denoted by $\Mk:=\Mk(\Rm,\mathcal{S})$. More specifically, the following action has been shown as an irreducible representation of $Spin(m)$:
\begin{eqnarray*}
\pi\ :\ Spin(m)\longrightarrow Aut(\Mk),\ s\longmapsto f(x)\mapsto \tilde{s}f(sx\tilde{s}).
\end{eqnarray*}
For more details, we refer the reader to \cite{M}.

%%%%%%%%                 Stein Weiss operators               %%%%%%%%
\subsubsection{Stein-Weiss operators}\hspace*{\fill} \\
\par
Let $U$ and $V$ be $m$-dimensional inner product vector spaces over a field $\mathbb{F}$. Denote the groups of all automorphism of $U$ and $V$ by $GL(U)$ and $GL(V)$, respectively. Suppose $\rho_1:\ G\longrightarrow GL(U)$ and $\rho_2:\ G\longrightarrow GL(V)$ are irreducible representations of a compact Lie group $G$. We have a function $f:\ U\longrightarrow V$ which has continuous derivative. Taking the gradient of the function $f(x)$, we have
$$\nabla f\in Hom(U,V)\cong U^*\otimes V\cong U\otimes V,\
 where\ \nabla:=(\partial_{x_1},\cdots,\partial_{x_m}).$$
Denote by $U[\times]V$ the irreducible representation of $U\otimes V$ whose representation space has largest dimension \cite{H}. This is known as the Cartan product of $\rho_1$ and $\rho_2$ \cite{Ea}. Using the inner products on $U$ and $V$, we may write
$$U\otimes V=(U[\times]V)\oplus(U[\times]V)^{\perp}$$
If we denote by $E$ and $E^{\perp}$ the orthogonal projections onto $U[\times]V$ and $(U[\times]V)^{\perp}$, respectively, then we define differential operators $D$ and $D^{\perp}$ associated to $\rho_1$ and $\rho_2$ by
$$D=E\nabla;\ D^{\perp}=E^{\perp}\nabla.$$
These are called \emph{Stein-Weiss type operators} after \cite{ES}. The importance of this construction is that you can reconstruct many first order differential operators with it when you choose proper representation spaces $U$ and $V$ for a Lie group $G$. For instance,  Euclidean Dirac operators \cite{Sh,ES} and Rarita-Schwinger operators \cite{G}. The connections are as follows:\\
\par
\emph{1. Dirac operators}
\hspace*{\fill} \\
\par
Here we only show the odd dimension case. Similar arguments also apply in the even dimensional case.
\begin{theorem}
Let $\rho_1$ be the representation of the spin group given by the standard representation of $SO(m)$ on $\Rm$
$$\rho_1:\ Spin(m)\longrightarrow SO(m)\longrightarrow GL(\Rm)$$
and let $\rho_2$ be the spin representation on the spinor space $\mathcal{S}$. Then the Euclidean Dirac operator is the differential operator given by $\Rm[\times]\mathcal{S}$ when $m=2n+1$.
\end{theorem}
\emph{Outline proof:} Let $\{e_1,\cdots,e_m\}$ be the orthonormal basis of $\Rm$ and $x=(x_1,\cdots,x_m)\in\Rm$. For a function $f(x)$ having values in $\scs$, we must show that the system
$$\sum_{i=1}^{m}e_i\displaystyle\frac{\partial f}{\partial x_i}=0$$
is equivalent to the system
$$D^{\perp}f=E^{\perp}\nabla f=0.$$
Since we have
$$\Rm\otimes\scs=\Rm[\times]\scs\ \oplus\ (\Rm[\times]\scs)^{\perp}$$
and \cite{ES} provides us an embedding map
\begin{eqnarray*}
&&\eta: \scs\hookrightarrow \Rm\otimes\scs,\\
&& \omega\mapsto \frac{1}{\sqrt{m}}(e_1\omega,\cdots,e_m\omega).
\end{eqnarray*}
Actually, this is an isomorphism from $\scs$ into $\Rm\otimes\scs$. For the proof, we refer the reader to \emph{page 175} of \cite{ES}. Thus, we have
$$\Rm\otimes\scs=\Rm[\times]\scs\ \oplus\ \eta(\scs).$$
Consider the equation $D^{\perp}f=E^{\perp}\nabla f=0$, where $f$ has values in $\scs$. So $\nabla f$ has values in $\Rm\otimes \scs$, and so the condition $D^{\perp}f=0$ is equivalent to $\nabla f$ being orthogonal to $\eta (\scs)$. This is precisely the statement that
$$\sum_{i=1}^{m}(\frac{\partial f}{\partial x_i},e_i\omega)=0,\ \forall\omega\in\scs.$$
Notice, however, that as an endomorphism of $\Rm\otimes\scs$, we have $-e_i$ as the dual of $e_i$, hence the equation above becomes
$$\sum_{i=1}^m(e_i\frac{\partial f}{\partial x_i},\omega)=0,\ \forall\omega\in\scs,$$
which says precisely that $f$ must be in the kernel of the Euclidean Dirac operator. This completes the proof.\qquad \qquad \qquad \qquad \qquad \qquad \qquad \qquad \qquad \qquad\qquad\qquad\qquad \qquad\quad \qedsymbol\\
\par
\emph{2. Rarita-Schwinger operators}
\hspace*{\fill} \\
\par
\begin{theorem}
Let $\rho_1$ be defined as above and $\rho_2$ is the representation of $Spin(m)$ on $\Mk$. Then as a representation of $Spin(m)$, we have the following decomposition
\begin{eqnarray*}
\Mk\otimes \Rm\cong \Mk[\times]\Rm\oplus\Mk\oplus\Mkk\oplus\mathcal{M}_{k,1},
\end{eqnarray*}
where $\mathcal{M}_{k,1}$ is a simplicial monogenic polynomial space as a $Spin(m)$ representation (see more details in \cite{Bie}). The Rarita-Schwinger operator is the differential operator given by projecting the gradient onto the $\Mk$ component.
\end{theorem}
\begin{proof}
Consider $f(x,u)\in C^{\infty}(\Rm,\mathcal{M}_k)$. We observe that the gradient of $f(x,u)$ satisfies
$$\nabla f(x,u)=(\partial_{x_1},\cdots,\partial_{x_m})f(x,u)=(\partial_{x_1}f(x,u),\cdots,\partial_{x_m}f(x,u))\in \mathcal{M}_k\otimes \Rm.$$
A similar argument as in \emph {page 181} of \cite{ES} shows 
$$\mathcal{M}_k\otimes \Rm=\Mk[\times]\Rm\oplus V_1\oplus V_2 \oplus V_3,$$
where $V_1\cong \Mk$, $V_2\cong \Mkk$ and $V_3\cong \mathcal{M}_{k,1}$ as $Spin(m)$ representations. Similar arguments as on \emph{page 175} of \cite{ES} show
\begin{eqnarray*}
\theta:\ \Mk\longrightarrow \Mk\otimes\Rm,\text{  }q_k(u)\mapsto (q_k(u)e_1,\cdots,q_k(u)e_m)
\end{eqnarray*} 
is an isomorphism from $\Mk$ into $\Mk\otimes\Rm$. Hence, we have
\begin{eqnarray*}
\mathcal{M}_k \otimes \Rm=\Mk[\times]\Rm\oplus\ \theta(\Mk)\oplus V_2\oplus V_3.
\end{eqnarray*}
Let $P'_k$ be the projection map from $\Mk \otimes \Rm$ to $\theta(\Mk)$. Consider the equation $P'_k\nabla f(x,u)=0$ for $f(x,u)\in C^{\infty}(\Rm,\Mk)$. Then, for each fixed $x$, $\nabla f(x,u)\in\Mk\otimes\Rm$ and the condition $P'_k\nabla f(x,u)=0$ is equivalent to $\nabla f$ being orthogonal to $\theta(\Mk)$. This says precisely
$$\sum_{i=1}^{m}(q_k(u)e_i,\partial_{x_i}f(x,u))_u=0,\ \forall q_k(u)\in\Mk,$$
where $(p(u),q(u))_u=\displaystyle\int_{\Sm}\overline{p(u)}q(u)dS(u)$ is the Fischer inner product for any pair of $\Clm$-valued polynomials. Since $-e_i$ is the dual of $e_i$ as an endomorphism of $\Mk\otimes\Rm$, the previous equation becomes
\begin{eqnarray*}
\sum_{i=1}^m(q_k(u),e_i\partial_{x_i}f(x,u))=(q_k(u),D_xf(x,u))_u=0.
\end{eqnarray*}
Since $f(x,u)\in\Mk$ for fixed $x$, then $D_xf(x,u)\in\Hk$. According to the Almansi-Fischer decomposition, we have
$$D_xf(x,u)=f_1(x,u)+uf_2(x,u), \text{  } f_1(x,u)\in\Mk \text{ and } f_2(x,u)\in\Mkk.$$
We then obtain $(q_k(u),f_1(x,u))_u+(q_k(u),uf_2(x,u))_u=0.$
However, the Clifford-Cauchy theorem \cite{D} shows
$(q_k(u),uf_2(x,u))_u=0.$
Thus, the equation $P'_k\nabla f(x,u)=0$ is equivalent  to $$(q_k(u),f_1(x,u))_u=0,\ \forall q_k(u)\in\Mk.$$
Hence, $f_1(x,u)=0$. We also know, from the construction of the Rarita-Schwinger operator, that $f_1(x,u)=R_kf(x,u)$. Therefore, the Stein-Weiss type operator $P'_k\nabla$ is precisely the Rarita-Schwinger operator in this context.
\end{proof}

\section{Properties of the Rarita-Schwinger operator}
%%%%%%%%%%            fact                 %%%%%%%%%%%%that
\subsection{A counterexample}\hspace*{\fill} \\
\par
We know that the Dirac operator $D_x$ is conformally invariant in $\mathcal{C}l_m$-valued function theory \cite{R1}. But in the Rarita-Schwinger setting, $D_x$ is not conformally invariant anymore. In other words, in $\mathcal{C}l_m$-valued function theory, the Dirac operator $D_x$ has the following conformal invariance property under inversion: If $D_xf(x)=0$, $f(x)$ is a $\Clm$-valued function and $x=y^{-1}$, $x\in \mathbb{R}^m$, then
$D_y\displaystyle\frac{y}{||y||^m}f(y^{-1})=0$.
In the Rarita-Schwinger setting, if $D_xf(x,u)=D_uf(x,u)=0$, $f(x,u)$ is a polynomial for any fixed $x\in \mathbb {R}^m$ and let $x=y^{-1},\ u=\displaystyle\frac{ywy}{||y||^2}$, $x\in \mathbb{R}^m$, then
$D_y\displaystyle\frac{y}{||y||^m}f(y^{-1},\displaystyle\frac{ywy}{||y||^2})\neq0$ in general.\\

A quick way to see this is to choose the function $f(x,u)=u_1e_1-u_2e_2$, and use $u=\displaystyle\frac{ywy}{||y||^2}=w-2\displaystyle\frac{y}{||y||^2}\langle w,y\rangle $, $u_i=w_i-2\displaystyle\frac{y_i}{||y||^2}\langle w,y\rangle$, where $i=1,2,\dots, m$. A straightforward calculation shows that
\begin{eqnarray*}
D_y\displaystyle\frac{y}{||y||^m}f(y^{-1},\frac{ywy}{||y||^2})=\frac{-2wy(y_1e_1-y_2e_2)}{||y||^{m+2}}\neq 0,
\end{eqnarray*}
for $m>2$. However, $P_1D_y\displaystyle\frac{y}{||y||^m}f(y^{-1},\frac{ywy}{||y||^2})=\big(\frac{wD_w}{m}+1\big)w\frac{-2y(y_1e_1-y_2e_2)}{||y||^{m+2}}=0.$

%%%%            Conformal invariance of Rarita-Schwinger operators               %%%%%%%%%%

\subsection{Conformal Invariance}\hspace*{\fill} \\
\par
In \cite{D}, the conformal invariance of the equation $R_kf=0$ is proved and some other properties under the assumption that $D_x$ is still conformally invariant in the Rarita-Schwinger setting. This is incorrect as we just showed. In this section, we will use the Iwasawa decomposition of M\"{o}bius transformations and some integral formulas to correct this. As observed earlier, according to this Iwasawa decomposition, a conformal transformation is a composition of translation, dilation, reflection and inversion. A simple observation shows that the  Rarita-Schwinger operator is conformally invariant under translation and dilation and the conformal invariance under reflection can be found in \cite{L}. Hence, we only show it is conformally invariant under inversion here.
\begin{theorem}\label{theorem Rk}
For any fixed $x\in U\subset\mathbb{R}^m$, let $f(x,u)$ be a left monogenic polynomial homogeneous of degree $k$ in $u$. If $R_{k,u}f(x,u)=0$, then $R_{k,w}G(y)f(y^{-1},\displaystyle\frac{ywy}{||y||^2})=0$, where $G(y)=\displaystyle\frac{y}{||y||^m},\ x=y^{-1},\ u=\displaystyle\frac{ywy}{||y||^2}\in \mathbb{R}^m$.
\end{theorem}
To establish the conformal invariance of $R_k$, we need $Stokes'\ Theorem$ for $R_k.$
\begin{theorem}[\cite{D}]\textbf{(Stokes' theorem for $R_k)$}\label{StokesRk}\\
Let $\Omega '$ and $\Omega$ be domains in $\mathbb{R}^m$ and suppose the closure of $\Omega$ lies in $\Omega '$. Further suppose the closure of $\Omega$ is compact and $\partial \Omega$ is piecewise smooth. Let $f,g\in C^1(\Omega ',\Mk)$. Then
\begin{eqnarray*}
&&\int_{\Omega}\big[(g(x,u)R_k,f(x,u))_u+(g(x,u),R_kf(x,u))\big]dx^m\\
=&&\int_{\partial\Omega}(g(x,u),P_kd\sigma_xf(x,u))_u\\
=&&\int_{\partial\Omega}(g(x,u)d\sigma_xP_{k,r},f(x,u))_u,
\end{eqnarray*}
where $P_k$ and $P_{k,r}$ are the left and right projections, $d\sigma_x=n(x)d\sigma(x)$, $d\sigma(x)$ is the area element. $(P(u),Q(u))_u=\int_{\mathbb{S}^{m-1}}P(u)Q(u)dS(u)$ is the inner product for any pair of $\Clm$-valued polynomials.
\end{theorem}

If both $f(x,u)$ and $g(x,u)$ are solutions of $R_k$, then we have \emph{Cauchy's theorem}.

\begin{corollary}[\cite{D}]\textbf{(Cauchy's theorem for $R_k)$}\\
If $R_kf(x,u)=0$ and $g(x,u)R_k=0$ for $f,g\in C^1(,\Omega ', \mathcal{M}_k)$, then
\begin{eqnarray*}
\int_{\partial\Omega}(g(x,u),P_kd\sigma_xf(x,u))_u=0.
\end{eqnarray*}
\end{corollary}
We also need the following well-known result.
\begin{proposition}[\cite{R}]\label{prop 1}
Suppose that $S$ is a smooth, orientable surface in $R^m$ and $f,\ g$ are integrable $\Clm$-valued functions. Then if $M(x)$ is a conformal transformation, we have
\begin{eqnarray*}
\int_Sf(M(x))n(M(x))g(M(x))ds=\int_{M^{-1}(S)}f(M(x))\tilde{J_1}(M,x)n(x)J_1(M,x)g(M(x))dM^{-1}(S),
\end{eqnarray*}
where $M(x)=(ax+b)(cx+d)^{-1}$, $M^{-1}(S)=\{x\in\mathbb{R}^m:M(x)\in S\}$, $J_1(M,x)=\displaystyle\frac{\widetilde{cx+d}}{||cx+d||^m}$.
\end{proposition}

Now we are ready to prove \emph{Theorem} \ref{theorem Rk}.
\begin{proof}
First, in Cauchy's theorem, we let $g(x,u)R_{k,r}=R_kf(x,u)=0$. Then we have
\begin{eqnarray*}
0=\int_{\partial\Omega}\int_{\mathbb{S}^{m-1}}g(x,u)P_kn(x)f(x,u)dS(u)d\sigma(x)
\end{eqnarray*}
Let $x=y^{-1}$, according to \emph{Proposition} \ref{prop 1}, we have
\begin{eqnarray*}
=\int_{\partial\Omega^{-1}}\int_{\mathbb{S}^{m-1}}g(u)P_{k,u}G(y)n(y)G(y)f(y^{-1},u)dS(u)d\sigma(y),
\end{eqnarray*}
where $G(y)=\displaystyle\frac{y}{||y||^m}$. Set $u=\displaystyle\frac{ywy}{||y||^2}$, since $P_{k,u}$ interchanges with $G(y)$ \cite{Li}, we have
\begin{eqnarray*}
&=&\int_{\partial\Omega^{-1}}\int_{\mathbb{S}^{m-1}}g(\displaystyle\frac{ywy}{||y||^2})G(y)P_{k,w}n(y)G(y)f(y^{-1},\displaystyle\frac{ywy}{||y||^2})dS(w)d\sigma(y)\\
&=&\int_{\partial\Omega^{-1}}(g(\displaystyle\frac{ywy}{||y||^2})G(y),P_{k,w}d\sigma_yG(y)f(y^{-1},\displaystyle\frac{ywy}{||y||^2}))_w,
\end{eqnarray*}
According to Stokes' theorem,
\begin{eqnarray*}
&=&\int_{\Omega^{-1}}(g(\displaystyle\frac{ywy}{||y||^2})G(y),R_{k,w}G(y)f(y^{-1},\displaystyle\frac{ywy}{||y||^2}))_w\\
&&+\int_{\Omega^{-1}}(g(\displaystyle\frac{ywy}{||y||^2})G(y)R_{k,w},G(y)f(y^{-1},\displaystyle\frac{ywy}{||y||^2}))_w.
\end{eqnarray*}
Since $g(x,u)$ is arbitrary in the kernel of $R_{k,r}$ and $f(x,u)$ is arbitrary in the kernel of $R_k$, we get $g(\displaystyle\frac{ywy}{||y||^2})G(y)R_{k,w}=R_{k,w}G(y)f(y^{-1},\displaystyle\frac{ywy}{||y||^2})=0$.
\end{proof}
%%%%     Intertwining operator of R_k      %%%%%%

\subsection{Intertwining operators of $R_k$}\hspace*{\fill} \\
\par
In $\mathcal{C}l_m$-valued function theory, if we have the M\"{o}bius transformation $y=\phi(x)=(ax+b)(cx+d)^{-1}$ and $D_x$ is the Dirac operator with respect to $x$ and $D_y$ is the Dirac operator with respect to $y$ then $D_x=J_{-1}^{-1}(\phi,x)D_yJ_1(\phi,x)$, where $J_{-1}(\phi,x)=\displaystyle\frac{cx+d}{||cx+d||^{m+2}}$ and $J_1(\phi,x)=\displaystyle\frac{\widetilde{cx+d}}{||cx+d||^m}$ \cite{R}. In the Rarita-Schwinger setting, we have a similar result:
\begin{theorem}[\cite{D}]\label{interwining}
For any fixed $x\in U\subset\mathbb{R}^m$, let $f(x,u)$ be a left monogenic polynomial homogeneous of degree $k$ in $u$. Then

\begin{eqnarray*}
J_{-1}^{-1}(\phi,y)R_{k,y,\omega}J_1(\phi,y)f(\phi(y),\displaystyle\frac{\widetilde{(cy+d)}\omega(cy+d)}{||cy+d||^2})=R_{k,x,u}f(x,u),
\end{eqnarray*}
where $x=\phi(y)=(ay+b)(cy+d)^{-1}$ is a M\"{o}bius transformation., $u=\displaystyle\frac{\widetilde{(cy+d)}\omega(cy+d)}{||cy+d||^2}$, $R_{k,x,u}$ and $R_{k,y,\omega}$ are Rarita-Schwinger operators.
\end{theorem}
\begin{proof}
We use the techniques in \cite{E} to prove this Theorem. Let $f(x,u),\ g(x,u)\in C^{\infty}(\Omega',\Clm)$ and $\Omega$ and $\Omega'$ are as in Theorem \ref{StokesRk}. We have
\begin{eqnarray*}
&&\int_{\partial\Omega}(g(x,u),P_kn(x)f(x,u))_udx^m\\
&=&\int_{\phi^{-1}(\partial\Omega)}\big(g(\phi(y),\frac{y\omega y}{||y||^2})P_kJ_1(\phi,y)n(y)J_1(\phi,y)f(\phi(y),\frac{y\omega y}{||y||^2})\big)_{\omega}dy^m\\
&=&\int_{\phi^{-1}(\partial\Omega)}\big(g(\phi(y),\frac{y\omega y}{||y||^2})J_1(\phi,y),P_kn(y)J_1(\phi,y)f(\phi(y),\frac{y\omega y}{||y||^2}))_{\omega}dy^m\\
\end{eqnarray*}
Then we apply the Stokes' Theorem for $R_k,$
\begin{eqnarray}\label{44}
&&\int_{\phi^{-1}(\Omega)}\big(g(\phi(y),\frac{y\omega y}{||y||^2})J_1(\phi,y)R_k,J_1(\phi,y)f(\phi(y),\frac{y\omega y}{||y||^2})\big)_{\omega} \nonumber \\
&+&\big(g(\phi(y),\frac{y\omega y}{||y||^2})J_1(\phi,y),R_kJ_1(\phi,y)f(\phi(y),\frac{y\omega y}{||y||^2})\big)_{\omega}dy^m,
\end{eqnarray}
where $u=\displaystyle\frac{y\omega y}{||y||^2}$.
On the other hand,
\begin{eqnarray}\label{55}
&&\int_{\partial\Omega}(g(x,u),P_kn(x)f(x,u))_udx^m\nonumber \\
&=&\int_{\Omega}\big[\big(g(x,u)R_k,f(x,u)\big)_u+\big(g(x,u),R_kf(x,u)\big)_u\big]dx^m \nonumber \\
&=&\int_{\phi^{-1}(\Omega)}\big[\big(g(x,u)R_k,f(x,u)\big)_u+\big(g(x,u),R_kf(x,u)\big)_u\big]j(y)dy^m \nonumber \\
&=&\int_{\phi^{-1}(\Omega)}\big[\big(g(x,u)R_k,f(x,u)j(y)\big)_u+\big(g(x,u),J_1(\phi,y)J_{-1}(\phi,y)R_kf(x,u)\big)_u\big]dy^m,
\end{eqnarray}
where $j(y)=J_{-1}(\phi,y)J_1(\phi,y)$ is the Jacobian. Now, we let arbitrary $g(x,u)\in ker R_{k,r}$ and since $J_1(\phi,y)g(\phi(y),\displaystyle\frac{y\omega y}{||y||^2})R_{k,r}=0$, then from (\ref{44}) and (\ref{55}), we get
\begin{eqnarray*}
&&\int_{\phi^{-1}(\Omega)}\big(g(\phi(y),\frac{y\omega y}{||y||^2})J_1(\phi,y)R_kJ_1(\phi,y)f(\phi(y),\frac{y\omega y}{||y||^2})\big)_{\omega}dy^m\\
&=&\int_{\phi^{-1}(\Omega)}\big(g(\phi(y),\frac{y\omega y}{||y||^2}),J_1(\phi,y)J_{-1}(\phi,y)R_kf(x,u)\big)_udy^m\\
&=&\int_{\phi^{-1}(\Omega)}\big(g(\phi(y),\frac{y\omega y}{||y||^2})J_1(\phi,y)J_{-1}(\phi,y)R_kf(x,u)\big)_{\omega}dy^m
\end{eqnarray*}
Since $\Omega$ is an arbitrary domain in $\Rm$, we have
$$\big(g(\phi(y),\frac{y\omega y}{||y||^2})J_1(\phi,y)R_kJ_1(\phi,y)f(\phi(y),\frac{y\omega y}{||y||^2})\big)_{\omega}=\big(g(\phi(y),\frac{y\omega y}{||y||^2})J_1(\phi,y)J_{-1}(\phi,y)R_kf(x,u)\big)_{\omega}$$
Also, $g(x,u)$ is arbitrary, we get
$$J_1(\phi,y)R_kJ_1(\phi,y)f(\phi(y),\frac{y\omega y}{||y||^2})=J_1(\phi,y)J_{-1}(\phi,y)R_kf(x,u).$$
Theorem \ref{interwining} follows immediately.
\end{proof}
%%%%%%%%%%%%% RS type operators   %%%%%%%%%%%%
\section{Rarita-Schwinger type operators}\hspace*{\fill} \\
\par
In the construction of the Rarita-Schwinger operator above, we notice that the Rarita-Schwinger operator is actually a projection map $P_k$ followed by the Dirac operator $D_x$, where in the Almansi-Fischer decomposition,
\begin{eqnarray*}
&&\Mk\xrightarrow{D_x}\Hk\otimes\mathcal{S}=\Mk\oplus u\Mkk\\
&&P_k:\ \Hk\otimes\mathcal{S}\longrightarrow\Mk;\\
&&I-P_k:\ \Hk\otimes\mathcal{S}\longrightarrow u\Mkk.
\end{eqnarray*}
If we project to the $u\Mkk$ component after we apply $D_x$, we get a Rarita-Schwinger type operator from $\Mk$ to $u\Mkk$.
\begin{eqnarray*}
\Mk\xrightarrow{D_x}\Hk\otimes\mathcal{S}\xrightarrow{I-P_k}u\Mkk.
\end{eqnarray*}
Similarly, starting with $u\Mkk$, we get another two Rarita-Schwinger type operators.
\begin{eqnarray*}
&&u\Mkk\xrightarrow{D_x}\Hk\otimes\mathcal{S}\xrightarrow {P_k}\Mk;\\
&&u\Mkk\xrightarrow{D_x}\Hk\otimes\mathcal{S}\xrightarrow{I-P_k}u\Mkk.
\end{eqnarray*}

In a summary, there are three further Rarita-Schwinger type operators as follows:

\begin{eqnarray*}
&&T_k^*:\ C^{\infty}(\Rm,\Mk)\longrightarrow C^{\infty}(\Rm,u\Mkk),\ \ \ T_k^*=(I-P_k)D_x=\frac{-uD_u}{m+2k-2}D_x;\\
&&T_k:\ C^{\infty}(\Rm,u\Mkk)\longrightarrow C^{\infty}(\Rm,\Mk),\ \ \ T_k=P_kD_x=(\frac{uD_u}{m+2k-2}+1)D_x;\\
&&Q_k:\ C^{\infty}(\Rm,u\Mkk)\longrightarrow C^{\infty}(\Rm,u\Mkk),\ \ \ Q_k=(I-P_k)D_x=\frac{-uD_u}{m+2k-2}D_x,
\end{eqnarray*}
$T_k^*$ and $T_k$ are also called the \emph{dual-twistor operator} and \emph{twistor operator}. See \cite{B}. We also have
\begin{eqnarray*}
&&T_{k,r}^*:\ C^{\infty}(\Rm,\overline{\mathcal{M}}_k)\longrightarrow C^{\infty}(\Rm,\overline{\mathcal{M}}_{k-1}u),\ \ \ T_{k,r}^*=D_x(I-P_{k,r});\\
&&T_{k,r}:\ C^{\infty}(\Rm,\overline{\mathcal{M}}_{k-1}u)\longrightarrow C^{\infty}(\Rm,\overline{\mathcal{M}}_k),\ \ \ T_k=D_xP_{k,r};\\
&&Q_{k,r}:\ C^{\infty}(\Rm,\overline{\mathcal{M}}_{k-1}u)\longrightarrow C^{\infty}(\Rm,\overline{\mathcal{M}}_{k-1}u),\ \ \ Q_k=D_x(I-P_{k,r}).
\end{eqnarray*}
%%%%%% Conformal invariance for Qk  %%%%%%%%%%%%%
\subsection{Conformal Invariance}\hspace*{\fill} \\
\par
We cannot prove conformal invariance and intertwining operators of $Q_k$ with the assumption that $D_x$ is conformally invariant. Here, we correct this using similar techniques that we used in Section 3 for the Rarita-Schwinger operators.\\
\par
Following our Iwasawa decomposition we only need to show the conformal invariance of $Q_k$ under inversion. We also need Cauchy's theorem for the $Q_k$ operator.
\begin{theorem}[\cite{Li}]\textbf{(Stokes' theorem for $Q_k$ operator)}\label{StokesQk}\\
Let $\Omega'$ and $\Omega$ be domains in $\mathbb{R}^m$ and suppose the closure of $\Omega$ lies in $\Omega'$. Further suppose the closure of $\Omega$ is compact and the boundary of $\Omega$, $\partial\Omega$ is piecewise smooth. Then for $f,\ g\in C^1(\Omega',\Mkk)$, we have
\begin{eqnarray*}
&&\int_{\Omega}[(g(x,u)uQ_{k,r},uf(x,u))_u+(g(x,u)u,Q_kuf(x,u))_u]dx^m\\
&=&\int_{\partial\Omega}(g(x,u)u,(I-P_k)d\sigma_xuf(x,u))_u\\
&=&\int_{\partial\Omega}(g(x,u)ud\sigma_x(I-P_{k,r}),uf(x,u))_u\\
\end{eqnarray*}
where $P_k$ and $P_{k,r}$ are the left and right projections, $d\sigma_x=n(x)d\sigma(x)$, $d\sigma(x)$ is the area element. $(P(u),Q(u))_u=\int_{\mathbb{S}^{m-1}}P(u)Q(u)dS(u)$ is the inner product for any pair of $\Clm$-valued polynomials.
\end{theorem}

When $g(x,u)uQ_{k,r}=Q_kuf(x,u)=0$, we get Cauchy's theorem for $Q_k$.

\begin{corollary}[\cite{Li}]\textbf{(Cauchy's theorem for $Q_k$ operator)}\\
If $Q_kuf(x,u)=0$ and $ug(x,u)Q_{k,r}=0$ for $f,g\in C^1(,\Omega ', \Mkk)$, then
\begin{eqnarray*}
\int_{\partial\Omega}(g(x,u)u,(I-P_k)d\sigma_xuf(x,u))_u=0
\end{eqnarray*}
\end{corollary}

The conformal invariance of the equation $Q_kuf=0$ under inversion is as follows

\begin{theorem}\label{theorem Qk}
For any fixed $x\in U\subset\mathbb{R}^m$, let $f(x,u)$ be a left monogenic polynomial homogeneous of degree $k-1$ in $u$. If $Q_{k,u}uf(x,u)=0$,  then $Q_{k,w}G(y)\displaystyle\frac{ywy}{||y||^2}f(y^{-1},\frac{ywy}{||y||^2})=0$, where $G(y)=\displaystyle\frac{y}{||y||^m},\ x=y^{-1},\ u=\frac{ywy}{||y||^2}\in \mathbb{R}^m$.
\end{theorem}

\begin{proof}
First, in Cauchy's theorem, we let $ug(x,u)Q_{k,r}=Q_kuf(x,u)=0$. Then we have
\begin{eqnarray*}
0=\int_{\partial\Omega}\int_{\mathbb{S}^{m-1}}g(u)u(I-P_k)n(x)uf(x,u)dS(u)d\sigma(x)
\end{eqnarray*}
Let $x=y^{-1}$, we have
\begin{eqnarray*}
=\int_{\partial\Omega^{-1}}\int_{\mathbb{S}^{m-1}}g(u)u(I-P_{k,u})G(y)n(y)G(y)uf(y^{-1},u)dS(u)d\sigma(y),
\end{eqnarray*}
where $G(y)=\displaystyle\frac{y}{||y||^m}$. Set $u=\displaystyle\frac{ywy}{||y||^2}$, since $I-P_{k,u}$ interchanges with $G(y)$ \cite{D}, we have
\begin{eqnarray*}
&=&\int_{\partial\Omega^{-1}}\int_{\mathbb{S}^{m-1}}g(\displaystyle\frac{ywy}{||y||^2})\frac{ywy}{||y||^2}G(y)(I-P_{k,w})n(y)G(y)\displaystyle\frac{ywy}{||y||^2}f(y^{-1},\displaystyle\frac{ywy}{||y||^2})dS(w)d\sigma(y)\\
&=&\int_{\partial\Omega^{-1}}\big(g(\displaystyle\frac{ywy}{||y||^2})\displaystyle\frac{ywy}{||y||^2}G(y),(I-P_{k,w})d\sigma_yG(y)\displaystyle\frac{ywy}{||y||^2}f(y^{-1},\displaystyle\frac{ywy}{||y||^2})\big)_w.
\end{eqnarray*}
According to Stokes' theorem for $Q_k$,
\begin{eqnarray*}
&=&\int_{\Omega^{-1}}\big(g(\displaystyle\frac{ywy}{||y||^2})\displaystyle\frac{ywy}{||y||^2}G(y),Q_{k,w}G(y)\displaystyle\frac{ywy}{||y||^2}f(y^{-1},\displaystyle\frac{ywy}{||y||^2})\big)_w\\
&&+\int_{\Omega^{-1}}\big(g(\displaystyle\frac{ywy}{||y||^2})\displaystyle\frac{ywy}{||y||^2}G(y)Q_{k,w},G(y)\displaystyle\frac{ywy}{||y||^2}f(y^{-1},\displaystyle\frac{ywy}{||y||^2})\big)_w.
\end{eqnarray*}
Since $ug(x,u)$ is arbitrary in the kernel of $Q_{k,r}$ and $uf(x,u)$ is arbitrary in the kernel of $Q_k$, we get $g(\displaystyle\frac{ywy}{||y||^2})\displaystyle\frac{ywy}{||y||^2}G(y)Q_{k,w}=Q_{k,w}G(y)\displaystyle\frac{ywy}{||y||^2}f(y^{-1},\displaystyle\frac{ywy}{||y||^2})=0$.
\end{proof}

To complete this section, we provide $Stokes'\ theorem$ for other Rarita-Schwinger type operators as follows:
\begin{theorem}\textbf{(Stokes' theorem for $T_k)$}\\
Let $\Omega '$ and $\Omega$ be domains in $\mathbb{R}^m$ and suppose the closure of $\Omega$ lies in $\Omega '$. Further suppose the closure of $\Omega$ is compact and $\partial \Omega$ is piecewise smooth. Let $f,g\in C^1(\Omega ',\Mk)$. Then
\begin{eqnarray*}
&&\int_{\Omega}\big[(g(x,u)T_k,f(x,u))_u+(g(x,u),T_kf(x,u))\big]dx^m\\
=&&\int_{\partial\Omega}(g(x,u),P_kd\sigma_xf(x,u))_u\\
=&&\int_{\partial\Omega}(g(x,u)d\sigma_xP_{k,r},f(x,u))_u,
\end{eqnarray*}
where $P_k$ and $P_{k,r}$ are the left and right projections, $d\sigma_x=n(x)d\sigma(x)$ and $(P(u),Q(u))_u=\int_{\mathbb{S}^{m-1}}P(u)Q(u)dS(u)$ is the inner product for any pair of $\mathcal{C}l_m$-valued polynomials.
\end{theorem}
\begin{theorem}\textbf{(Stokes' theorem for $T_k^* )$}\\
Let $\Omega '$ and $\Omega$ be domains in $\mathbb{R}^m$ and suppose the closure of $\Omega$ lies in $\Omega '$. Further suppose the closure of $\Omega$ is compact and $\partial \Omega$ is piecewise smooth. Let $f,g\in C^1(\Omega ',u\Mkk)$. Then
\begin{eqnarray*}
&&\int_{\Omega}\big[(g(x,u)T_k^*,f(x,u))_u+(g(x,u),T_k^*f(x,u))\big]dx^m\\
=&&\int_{\partial\Omega}(g(x,u),(I-P_k)d\sigma_xf(x,u))_u\\
=&&\int_{\partial\Omega}(g(x,u)d\sigma_x(I-P_{k,r}),f(x,u))_u,
\end{eqnarray*}
where $P_k$ and $P_{k,r}$ are the left and right projections, $d\sigma_x=n(x)d\sigma(x)$ and $(P(u),Q(u))_u=\int_{\mathbb{S}^{m-1}}P(u)Q(u)dS(u)$ is the inner product for any pair of $\mathcal{C}l_m$-valued polynomials.
\end{theorem}
\begin{theorem}\textbf{(Alternative form of Stokes' Theorem)}\\
Let $\Omega$ and $\Omega '$ be as in the previous theorem. Then for $f\in C^1(\Rm,\Mk)$ and $g\in C^1(\Rm,\Mkk)$, we have
\begin{eqnarray*}
&&\int_{\partial\Omega}\big(g(x,u)ud\sigma_xf(x,u)\big)_u\\
&=&\int_{\Omega}\big(g(x,u)uT_k,f(x,u)\big)_udx^m+\int_{\Omega}\big(g(x,u)u,T_k^*f(x,u)\big)_udx^m.
\end{eqnarray*}
Further
\begin{eqnarray*}
&&\int_{\partial\Omega}\big(g(x,u)ud\sigma_xf(x,u)\big)_u\\
&=&\int_{\partial\Omega}\big(g(x,u)u,(I-P_k)d\sigma_xf(x,u)\big)_u\\
&=&\int_{\partial\Omega}\big(g(x,u)ud\sigma_xP_k,f(x,u)\big)_u.
\end{eqnarray*}
\end{theorem}

%%%%%%%%%%%%%%     Reference     %%%%%%%%%%%%

\end{document}